\documentclass[12pt]{article}
\usepackage{amssymb}
\usepackage{amsthm}
\usepackage{amsmath}
\usepackage{graphicx}
 \newtheorem{thm}{Theorem}[section]
 \newtheorem{cor}[thm]{Corollary}
 \newtheorem{lem}[thm]{Lemma}
 
 \theoremstyle{definition}
 
 \newtheorem{rem}[thm]{Remark}
 \numberwithin{equation}{section}

\begin{document}
\title{Remarks on the Blowup Criteria for Oldroyd Models}
\author{Zhen Lei\footnote{Shanghai Key Laboratory for Contemporary Applied
 Mathematics; School of Mathematical Sciences, Fudan
  University, Shanghai 200433, China. {\it Email:
  leizhn@yahoo.com, zlei@fudan.edu.cn}}\and Nader Masmoudi\footnote{Courant Institute of
  Mathematical Sciences, New York University, NY 10012, USA.
  {\it Email: masmoudi@cims.nyu.edu
}}\and
  Yi Zhou\footnote{Shanghai Key Laboratory for Contemporary Applied
 Mathematics; School of Mathematical Sciences, Fudan
  University, Shanghai 200433, China. {\it Email: yizhou@fudan.ac.cn}}}
\date{\today}
\maketitle

\begin{abstract}
We provide a new method to prove and improve the Chemin-Masmoudi
criterion for viscoelastic systems of Oldroyd type in \cite{CM} in
two space dimensions. Our method is much easier than the one based
on the well-known \textit{losing a priori estimate} and is
expected to be easily adopted to other problems involving the
losing \textit{a priori} estimate.
\end{abstract}


\section{Introduction}%----------------------------introduction

In this paper, we are going to study the non-blowup criteria of
solutions of a type of incompressible non-Newtonian fluid flows
described by the Oldroyd-B model in the whole 2D space:
\begin{equation}\label{a1}
\begin{cases}
\partial_tv + v\cdot\nabla v + \nabla p = \nu\Delta v + \mu_1\nabla\cdot\tau,\\[-3mm]\\
\partial_t\tau + v\cdot\nabla \tau + a\tau = Q(\tau, \nabla v) + \mu_2 D(v),
\\[-3mm]\\
\nabla\cdot v = 0,
\end{cases}
\end{equation}
where $v$ is the velocity field, $\tau$ is the non-Newtonian part
of the stress tensor and $p$ is the pressure. The constants $\nu$
(the viscosity of the fluid), $a$ (the reciprocal of the
relaxation time), $\mu_1$ and $\mu_2$ (determined by the dynamical
viscosity of the fluid, the retardation time and $a$) are assumed
to be non-negative. The bilinear term $Q$ has the following form:
\begin{eqnarray}\label{a2}
Q(\tau, \nabla v) = W(v)\tau - \tau W(v) + b\big(D(v)\tau + \tau
D(v)\big).
\end{eqnarray}
Here $b \in [-1, 1]$ is a constant, $D(v) = \frac{\nabla v +
(\nabla v)^t}{2}$ is the deformation tensor and $W(v) =
\frac{\nabla v - (\nabla v)^t}{2}$ is the vorticity tensor. Fluids
of this type have both elastic properties and viscous properties.
More discussions and the derivation of Oldroyd-B model \eqref{a1}
can be found in Oldroyd \cite{Oldroyd} or Chemin and Masmoudi
\cite{CM}.

There has been a lot of work on the existence theory of Oldroyd
model \cite{GS-1, GS-2, FG, LM, CM, Lei}. In particular, the
following Theorem is established by Chemin and Masmoudi in
\cite{CM}:

\bigskip

\textit{Theorem (Chemin and Masmoudi): In two space dimensions,
the solutions to the Oldroyd model \eqref{a1} with smooth initial
data do not develop singularities for $t \leq T$ provided that
\begin{eqnarray}\label{a3}
 \int_0^{T}\|\tau(t, \cdot)\|_{L^\infty} +
|b| \ \|\tau(t, \cdot)\|_{L^2}^2 dt < \infty.
\end{eqnarray}
}

\bigskip

To establish the blowup criterion \eqref{a3}, the authors in
\cite{CM} use a losing \textit{a priori} estimate for solutions of
transport equations which was developed by Bahouri and Chemin
\cite{BC} and used later on by a lot of authors (for example, see
\cite{CL, CM, ConM,  LZZ, MZZ08, LM07,Masmoudi08cpam} and the
references therein). Our purpose of this paper is to provide a
simple method which avoids using the complicated losing \textit{a
priori} estimate and to improve the blowup criterion \eqref{a3}
for Oldroyd model \eqref{a1} established by Chemin and Masmoudi
\cite{CM}. To best illustrate our ideas and for simplicity, we
will take $a = 0$ and $\nu = \mu_1 = \mu_2 = b = 1$ throughout
this paper. More precisely, we study the following system
\begin{equation}\label{a4}
\begin{cases}
\partial_tv + v\cdot\nabla v + \nabla p = \Delta v + \nabla\cdot\tau,\\[-4mm]\\
\partial_t\tau + v\cdot\nabla \tau = \nabla v\tau +  \tau(\nabla v)^t + D(v),
\\[-4mm]\\
\nabla\cdot v = 0.
\end{cases}
\end{equation}
We point out here that the results in this paper are obviously
true for general constants $a, \mu_1, \mu_2 \geq 0$, $\nu, b
> 0$ from our proofs.

Our main result concerning system \eqref{a4} is:
\begin{thm}\label{thm-criterion-Oldroyd}
Assume that $(v, \tau)$ is a local smooth solution to the Oldroyd
model \eqref{a4} on $[0, T)$ and    $\|v(0, \cdot)\|_{L^2 \cap \dot{C}^{1 +
\alpha} (\mathbb{R}^2) } + \|\tau(0, \cdot)\|_{ L^1 \cap \dot{C}^\alpha (\mathbb{R}^2)} < \infty$ for some
$\alpha \in (0, 1)$. Then one has
\begin{eqnarray}\nonumber
\|v(t, \cdot)\|_{\dot{C}^{1 + \alpha}} + \|\tau(t,
\cdot)\|_{\dot{C}^\alpha} < \infty
\end{eqnarray}
for all $0 \leq t \leq T$ provided that
\begin{eqnarray}\label{a5}
\|\tau(t, \cdot)\|_{L^1_T({\rm BMO})} < \infty \quad {and}\
\|\tau\|_{L^\infty_T(L^1)} < \infty.
\end{eqnarray}
\end{thm}

\begin{rem}
This result is in the spirit of  the Beale-Kato-Majda \cite{BKM}
non blowup criterion for $3D$ Euler equations. There were many
subsequent results improving the criterion, see for instance
\cite{Kozonoaniuchi,KOT02,Planchon03cmp}.
% for related result.
%  From our proof, we can see that the  first
In particular our result still holds if we replace $BMO$ with the
Besov space $B^0_{\infty, \infty}$ used in H.~Kozono, T.~Ogawa,
and Y.~Taniuchi. \cite{KOT02} or if we replace  the condition by
the one introduced in Planchon \cite{Planchon03cmp}. In other
words, the first condition in \eqref{a5} in the above theorem can
be weakened to
\begin{eqnarray}\nonumber
  \int_{0}^{T} \| \tau \|_{B^0_{\infty, \infty}}    =  \int_{0}^{T}
 \sup_q   \|\Delta_q\tau(t, \cdot)\|_{L^\infty}dt < \infty,
\end{eqnarray}
or to
\begin{eqnarray}\nonumber
\lim_{\delta\rightarrow 0}\sup_q  \int_{T - \delta}^{T}
\|\Delta_q\tau(t, \cdot)\|_{L^\infty}dt < \epsilon,
\end{eqnarray}
for some sufficiently small $\epsilon > 0$.
%  We do not pursue this
% question here.
The second condition in \eqref{a5} can be replaced by
\begin{eqnarray}\nonumber
\|\tau\|_{L^2_T(L^2)} < \infty,
\end{eqnarray}
which was used in \cite{CM}.
\end{rem}

\begin{rem}
It is easy to check that smooth solutions to \eqref{a4} enjoy the
following energy law:
\begin{eqnarray}\label{energy}
\int_{\mathbb{R}^2}|v(t, \cdot)|^2 + {\rm tr}\tau(t, \cdot) dx +
\int_0^t\int_{\mathbb{R}^2}|\nabla v|^2dxds =
\int_{\mathbb{R}^2}|v(0, \cdot)|^2 + {\rm tr}\tau(0, \cdot) dx,
\end{eqnarray}
which means that
\begin{eqnarray}\label{aprioriestimate}
v \in L^\infty_T(L^2) \cap L^2_T(\dot{H}^1)
\end{eqnarray}
for all $T > 0$ under the second condition of \eqref{a5}. The
\textit{a priori} estimate \eqref{aprioriestimate} will be important
to apply lemma \ref{lem-CM}.
\end{rem}

Finally, it is well-known that if $A =  2 \tau +  I$ is
a positive definite symmetric matrix at $t=0$ (which is actually the physical case),
 then this  property is conserved for later times. Indeed, $A$ satisfies the
equation
$$ \partial_t A  + v\cdot\nabla A  = \nabla v A  +  A (\nabla v)^t .  $$
Also, if at $t=0$, we have  $\det(A(0)) > 1$ and $A$ is positive
definite, then this will also hold for later times (see
\cite{HL07}). In particular this implies that  tr$(\tau) > 0$ (or
one has $- 1 < {\rm tr}(\tau) \leq 0$, which contradicts with
$\det(A) > 1$). Hence, we have the following corollary where we also
use the improved criterion of Planchon.

\begin{cor}\label{cor-criterion-Oldroyd}
There exists an $\epsilon > 0$, such that if
 $(v, \tau)$ is a local smooth solution to the Oldroyd
model \eqref{a4} on $[0, T)$,     $\|v(0, \cdot)\|_{L^2 \cap \dot{C}^{1 +
\alpha} (\mathbb{R}^2) } + \|\tau(0, \cdot)\|_{ L^1 \cap \dot{C}^\alpha (\mathbb{R}^2)} < \infty$ for some
$\alpha \in (0, 1)$ and  that det$(I + 2 \tau(0)) > 1, \,  A = I + 2 \tau(0) $
is positive definite symmetric.
Then one has
\begin{eqnarray}\nonumber
\|v(t, \cdot)\|_{\dot{C}^{1 + \alpha}} + \|\tau(t,
\cdot)\|_{\dot{C}^\alpha} < \infty
\end{eqnarray}
for all $0 \leq t \leq T$ provided that
\begin{eqnarray}\label{a5-cor}
\lim_{\delta\rightarrow 0}\sup_q  \int_{T - \delta}^{T}
\|\Delta_q\tau(t, \cdot)\|_{L^\infty}dt < \epsilon.
% \|\tau(t, \cdot)\|_{L^1_T({\rm BMO})} < \infty.
% \quad {and}\
% \|\tau\|_{L^\infty_T(L^1)} < \infty.
\end{eqnarray}
\end{cor}

Our proof is based on careful H${\rm \ddot{o}}$lder estimates of
heat and transport equations and the standard Littlewood-Paley
theory, which is much easier than the extensively used losing
\textit{a priori} estimates (for example, see \cite{BC, CL, CM,
ConM, LZZ}). In fact, the main innovation of this paper is that
our analysis may be viewed as a replacement of the losing
\textit{a priori} estimate.  Our method is expected to be easily
adopted to other problems via the losing \textit{a priori}
estimate. Moreover, our criterion slightly improves the one
established by Chemin and Masmoudi (see \cite{CM}).

Finally, let us make a remark on MHD:
\begin{equation}\label{MHD}
\begin{cases}
\partial_tv + v\cdot\nabla v + \nabla p = \nu\Delta v + \nabla\cdot(H\times H),\\[-4mm]\\
\partial_tH + v\cdot\nabla H = H\cdot\nabla v, \\[-4mm]\\
\nabla\cdot v = \nabla\cdot H = 0,
\end{cases}
\end{equation}
where $H$ denotes the magnetic field. A direct corollary of Theorem
\ref{thm-criterion-Oldroyd} for MHD \eqref{MHD} is the following:
\begin{cor}\label{thm-criterion-MHD}
Assume that $(v, H)$ is a local smooth solution to MHD \eqref{MHD}
on $[0, T)$ and $\|v(0, \cdot)\|_{L^2\cap\dot{C}^{1 + \alpha}} +
\|H(0, \cdot)\|_{L^2\cap\dot{C}^\alpha} < \infty$ for some $\alpha
\in (0, 1)$. Then one has
\begin{eqnarray}\nonumber
\|v(t, \cdot)\|_{\dot{C}^{1 + \alpha}} + \|H(t,
\cdot)\|_{\dot{C}^\alpha} < \infty
\end{eqnarray}
all $0 \leq t \leq T$ provided that
\begin{eqnarray}\label{a9}
\int_0^{T}\|(H\times H)(t, \cdot)\|_{{\rm BMO}}dt < \infty.
\end{eqnarray}
\end{cor}
The proof of this corollary is given in section 4. Unfortunately, at
present we are not able to improve \eqref{a9} as
\begin{eqnarray}\nonumber
\int_0^{T}\|H(t, \cdot)\|_{{\rm BMO}}^2dt < \infty,
\end{eqnarray}
and this is still an open problem.

The paper is organized as follows: Section 2 is devoted to recalling
some basic properties of Littlewood-Paley theory and proving two
interpolation inequalities. The proof of Theorem
\ref{thm-criterion-Oldroyd} is given in section 3. In the last
section we sketch the proof of Corollary \ref{thm-criterion-MHD}.

\section{Preliminaries}

Let $\mathcal{S}(\mathbb{R}^2)$ be the Schwartz class of rapidly
decreasing functions. Given $f \in \mathcal{S}(\mathbb{R}^2)$, its
Fourier transform $\mathcal{F}f = \hat{f}$ (inverse Fourier
transform $\mathcal{F}^{-1}g = \breve{g}$, respectively) is
defined by $\hat{f}(\xi) = \int e^{-ix\cdot\xi}f(x)dx$
($\breve{g}(x) = \int e^{ix\cdot\xi}g(\xi)d\xi$, respectively).
Now let us recall the Littlewood-Paley decomposition (see
\cite{Bony, Chemin}). Choose two nonnegative radial functions
$\psi, \phi \in \mathcal{S}(\mathbb{R}^2)$, supported respectively
in $B = \{\xi \in \mathbb{R}^2: |\xi| \leq \frac{4}{3}$ and $C =
\{\xi \in \mathbb{R}^2: \frac{3}{4} \leq |\xi| \leq \frac{8}{3}$
such that
\begin{equation}\nonumber
\psi(\xi) + \sum_{j \geq 0}\phi(\frac{\xi}{2^j}) = 1\ {\rm for}\
\xi \in R^2,\quad \sum_{- \infty \leq j \leq
\infty}\phi(\frac{\xi}{2^j}) = 1\ {\rm for}\ \xi \in R^2/\{0\}.
\end{equation}
The frequency localization operator is defined by
\begin{equation}\label{b3}
\Delta_qf = \int_{\mathbb{R}^2}\check{\phi}(y)f(x - 2^{-
q}y)dy,\quad S_qf = \int_{\mathbb{R}^2}\check{\psi}(y)f(x - 2^{-
q}y)dy.
\end{equation}

The following Lemma is well-known (for example, see \cite{Chemin}).
\begin{lem}\label{lem2}
For $s \in \mathbb{R}$, $1 \leq p \leq \infty$ and integer $q$,
one has
\begin{equation}\label{b2}
\begin{cases}
c2^{qs}\|\Delta_qf\|_{L^p} \leq \|\nabla^s\Delta_qf\|_{L^p} \leq
  C2^{qs}\|\Delta_qf\|_{L^p},\\[-4mm]\\
\big\||\nabla|^sS_qf\big\|_{L^p} \leq  C2^{qs}\|f\|_{L^p},\\[-4mm]\\
ce^{- C2^{2q}t}\|\Delta_qf\|_{L^\infty} \leq
  \|e^{t\Delta}\Delta_qf\|_{L^\infty} \leq  Ce^{-
  c2^{2q}t}\|\Delta_qf\|_{L^\infty}.
\end{cases}
\end{equation}
Here $C$ and $c$ are positive constants independent of $s$, $p$
and $q$.
\end{lem}

We also need the following Lemma (see also \cite{MZZ08,Masmoudi09prep}
where similar estimates were used.)
\begin{lem}\label{lem1}
Assume that $\beta > 0$. Then there exists a positive constant $C
> 0$ such that
\begin{equation}\label{b1}
\begin{cases}
\|f\|_{L^\infty} \leq C\big(1 + \|f\|_{L^2} + \|f\|_{{\rm
  BMO}}\ln(e + \|f\|_{\dot{C}^\beta})\big),\\[-4mm]\\
%\|f\|_{L^\infty} \leq C\big(1 + \|f\|_{L^2} + \|\nabla f\|_{L^2}
%\ln^{\frac{1}{2}}(1 + \|f\|_{\dot{C}^\beta})\big),\\[-4mm]\\
\int_0^T\|\nabla g(s, \cdot)\|_{L^\infty}ds \leq
  C\Big(1 + \int_0^T\|g(s, \cdot)\|_{L^2}ds\\
\quad +\ \sup_{q}\int_0^T\|\Delta_q\nabla g(s,
  \cdot)\|_{L^\infty}ds\ln\big( e   + \int_0^T    \|\nabla
  g(s, \cdot)\|_{\dot{C}^\beta}ds\big)\Big).
\end{cases}
\end{equation}
\end{lem}

\begin{proof}

The first inequality is well-known. For example, see \cite{BKM,
Kozonoaniuchi, LZ}.
%To prove the second inequality, we use the
%standard Littlewood-Paley theory to estimate that
%\begin{eqnarray}\nonumber
%&&\|f\|_{L^\infty}ds \leq \|\sum_{j \leq 0}\Delta_jf\|_{L^\infty}
%  + \|\sum_{j = 1}^N\Delta_jf\|_{L^\infty} + \|\sum_{j \geq N + 1}
%  \Delta_jf\|_{L^\infty}\\\nonumber
%&&\leq C\Big(\|f\|_{L^2} + \sum_{j = 1}^N\|\Delta_j\nabla
%  f\|_{L^2} + 2^{-\beta N}\|f\|_{\dot{C}^\beta}\Big)\\\nonumber
%&&\leq C\Big(\|f\|_{L^2} + N^{\frac{1}{2}}\big(\sum_{j =
%  1}^N\|\Delta_j\nabla f\|_{L^2}^2\big)^{\frac{1}{2}}
%  + 2^{-\beta N}\|f\|_{\dot{C}^\beta}\Big)\\\nonumber
%&&\leq C\Big(\|f\|_{L^2} + N^{\frac{1}{2}}\|\nabla f\|_{L^2}
%  + 2^{-\beta N}\|f\|_{\dot{C}^\beta}\Big).
%\end{eqnarray}
%Choosing $N$ such that $2^{-\beta N}\|f\|_{\dot{C}^\beta} = 1$ gives
%the desire inequality.
To prove the second inequality, we use the
Littlewood-Paley theory to compute that
\begin{eqnarray}\nonumber
&&\int_0^T\|\nabla g(s, \cdot)\|_{L^\infty}ds \leq
  C\int_0^T\|\sum_{q \leq 0}\nabla \Delta_qg(s,
  \cdot)\|_{L^\infty}ds\\\nonumber
&&\quad +\ CN\max_{1 \leq q \leq N}\int_0^T\|\Delta_q\nabla g(s,
  \cdot)\|_{L^\infty}ds + \int_0^T\sum_{q \geq N +
  1}2^{-\beta q}2^{\beta q}\|\Delta_q\nabla g(s, \cdot)\|_{L^\infty}ds\\\nonumber
&&\leq C\Big(\int_0^T\|g(s, \cdot)\|_{L^2}ds + \sup_{1 \leq q \leq
  N}\int_0^TN\|\Delta_q\nabla g(s, \cdot)\|_{L^\infty}ds\\\nonumber
&&\quad +\ 2^{- \beta N}\int_0^T\|\nabla g(s, \cdot)
  \|_{\dot{C}^\beta}ds\Big).
\end{eqnarray}
Then the second inequality in Lemma \ref{lem1} follows by choosing
\begin{equation}\nonumber
N = \frac{1}{\beta}\log_2\big(e +  \int_0^T  \|\nabla g(s,
\cdot)\|_{\dot{C}^\beta}\big) \leq C\ln\big(e  + \int_0^T   \|\nabla
g(s, \cdot)\|_{\dot{C}^\beta}ds\big).
\end{equation}
\end{proof}

\section{Blowup Criteria for Oldroyd-B Model}

This section is devoted to establishing the blowup criterion for
the Oldroyd-B Model \eqref{a1} and proving Theorem
\ref{thm-criterion-Oldroyd}. Our analysis is based on careful
H${\rm \ddot{o}}$lder estimates of heat and transport equations
and the standard Littlewood-Paley theory, which is much easier
than the extensively used losing \textit{a priori} estimates (for
example, see \cite{BC, CL, CM, ConM, LZZ}). Moreover, our
criterion slightly improves the one established by Chemin and
Masmoudi (see \cite{CM}). We divide our proof into two steps. The
first step is focused on establishing some \textit{a priori}
estimates for 2-D Navier-Stokes equations. Then we establish
H${\rm \ddot{o}}$lder estimates for the velocity field $v$ and the
stress tensor $\tau$ in the second step.

\bigskip

\textbf{Step 1. The \textit{a priori} estimates for 2-D
Navier-Stokes equations}.

\bigskip

We need the following Lemma which is basically established by
Chemin and Masmoudi in \cite{CM}. For completeness, the proof will
be also sketched here.

\begin{lem}\label{lem-CM}(Chemin-Masmoudi)
Let $v$ be a a solution of the Navier-Stokes equations with
initial data in $L^2$ and an external force $f \in
\widetilde{L}^1_T(C^{-1}) \cap L^2_T(H^{-1})$:
\begin{equation}\label{c1}
\begin{cases}
\partial_tv + v\cdot\nabla v + \nabla p = \Delta v + f,\\[-4mm]\\
\nabla\cdot v = 0,\\[-4mm]\\
v(0, x) = v_0(x).
\end{cases}
\end{equation}
Then we have the following a priori estimate:
\begin{eqnarray}\label{c2}
&&\|v\|_{\widetilde{L}^1_T(C^1)} \leq
  C\Big(\sup_q\|\Delta_qv_0\|_{L^2}\big(1 - \exp\{-c2^{2q}T\}\big)\\\nonumber
&&\quad +\ \big(\|v_0\|_{L^2} + \|f\|_{L^2_T(\dot{H}^{-1})}\big)
  \|\nabla v\|_{L^2_T(L^2)}^2 + \sup_q\int_0^T\|2^{-q}\Delta_qf(s)\|_{L^\infty}ds
  \Big).
\end{eqnarray}
\end{lem}

\begin{proof}
First of all, applying the operator $\Delta_q$ to the 2-D
Navier-Stokes equations \eqref{c1} and then using Lemma \ref{lem2}
and the standard energy estimate, we deduce that
\begin{eqnarray}\nonumber
&&\frac{1}{2}\frac{d}{dt}\|\Delta_qv\|_{L^2}^2 +
  c2^{2q}\|\Delta_qv\|_{L^2}^2\\\nonumber
&&\leq \|2^q\Delta_qv\|_{L^2}\big(\|2^{-q}\Delta_qf\|_{L^2} +
  \|\Delta_q(v \otimes v)\|_{L^2}\big)\\\nonumber
&&\leq c2^{2q}\|\Delta_qv\|_{L^2}^2 + C\big(\|2^{-q}
  \Delta_qf\|_{L^2}^2 + \|\Delta_q(v \otimes v)\|_{L^2}^2\big).
\end{eqnarray}
Integrating with respect to time and summing over $q$, we get
\begin{eqnarray}\nonumber
&&\sum_q\|\Delta_qv\|_{L^\infty_T(L^2)}^2 \leq
  \|v_0\|_{L^2}^2 + C\big(\|f\|_{L^2_T(\dot{H}^{-1})}^2
  + \|v \otimes v\|_{L^2_T(L^2)}^2\big)\\\nonumber
&&\leq \|v_0\|_{L^2}^2 + C\big(\|f\|_{L^2_T(\dot{H}^{-1})}^2 +
  \|v\|_{L^\infty_T(L^2)}^2\|\nabla v\|_{L^2_T(L^2)}^2\big),
\end{eqnarray}
where we used the standard interpolation inequality $\|v\|_{L^4}^2
\leq C\|v\|_{L^2}\|\nabla v\|_{L^2}$. Recalling the basic energy
estimate
\begin{eqnarray}\label{c3}
\|v\|_{L^\infty_T(L^2)}^2 + \|\nabla v\|_{L^2_T(L^2)}^2 \leq
\|v_0\|_{L^2}^2 + \|f\|_{L^2_T(\dot{H}^{-1})}^2,
\end{eqnarray}
one has
\begin{eqnarray}\label{c4}
&&\sum_q\|\Delta_qv\|_{L^\infty_T(L^2)}^2 \leq
  C\big(\|v_0\|_{L^2}^2 + \|f\|_{L^2_T(\dot{H}^{-1})}^2\big)\\\nonumber
&&\quad \times\big(1 + \|v_0\|_{L^2}^2 +
\|f\|_{L^2_T(\dot{H}^{-1})}^2\big).
\end{eqnarray}

Next, let us apply $\Delta_q$ to \eqref{c1} and use Lemma
\ref{lem2} to estimate
\begin{eqnarray}\nonumber
&&\|\Delta_qv(t)\|_{L^\infty} \leq C\|\Delta_q
  v_0\|_{L^\infty}e^{-c2^{2q}t}\\\nonumber
&&\quad +\ \int_0^t\big(\|\Delta_qf(s)\|_{L^\infty} +
  \|\Delta_q\nabla\cdot(v \otimes v)(s)\|_{L^\infty} \big)
  e^{-c2^{2q}(t - s)}ds,
\end{eqnarray}
which yields
\begin{eqnarray}\label{c5}
&&\|v\|_{\widetilde{L}^1_T(C^1)} \leq C\sup_q\int_0^T\|\Delta_q
  v_0\|_{L^\infty}2^qe^{-c2^{2q}t}dt\\\nonumber
&&\quad +\ C\sup_q\int_0^T\int_0^t\|\Delta_qf(s)
  \|_{L^\infty}2^qe^{-c2^{2q}(t - s)}dsdt\\\nonumber
&&\quad +\ C\sup_q\int_0^T\int_0^t\|\Delta_q\nabla\cdot(v \otimes
  v)(s)\|_{L^\infty}2^qe^{-c2^{2q}(t - s)}dsdt\\\nonumber
&&\leq C\sup_q\|\Delta_q v_0\|_{L^2}\big(1 -
  e^{-c2^{2q}T}\big)\\\nonumber
&&\quad +\ C\sup_q\int_0^T\|\Delta_q(v \otimes
  v)(s)\|_{L^\infty}ds + \|f\|_{\widetilde{L}^1_T
  (C^{-1})}.
\end{eqnarray}
Using  the Bony's decomposition, one can write
\begin{eqnarray}\nonumber
&&\|\Delta_q(v \otimes v)(s)\|_{L^\infty} =
  \sum_{|p - r| \leq 1}\|\Delta_q(\Delta_pv \otimes
  \Delta_rv)(s)\|_{L^\infty}\\\nonumber
&&\quad +\ \sum_{p - r \geq 2}\|\Delta_q(\Delta_pv
  \otimes \Delta_rv)(s)\|_{L^\infty}
  + \sum_{r - p \geq 2}\|\Delta_q(\Delta_pv \otimes
  \Delta_rv)(s)\|_{L^\infty}.
\end{eqnarray}
A straightforward computation gives
\begin{eqnarray}\nonumber
&&\int_0^T\sum_{|p - r| \leq 1}\|\Delta_q(\Delta_pv \otimes
  \Delta_rv)(s)\|_{L^\infty}ds\\\nonumber
&&\leq C\int_0^T\sum_{|p - r| \leq 1}2^q\|\Delta_q(\Delta_pv
  \otimes \Delta_rv)(s)\|_{L^2}ds\\\nonumber
&&\leq C\int_0^T\sum_{|p - r| \leq 1,\ p \geq q - 3}
  2^{q - \frac{p + r}{2}}\|2^p\Delta_pv\|_{L^\infty}^{\frac{1}{2}}
  \|\Delta_rv\|_{L^2}^{\frac{1}{2}}\|\Delta_pv\|_{L^\infty}^{\frac{1}{2}}
  \|2^r\Delta_rv\|_{L^2}^{\frac{1}{2}}ds\\\nonumber
&&\leq C\int_0^T\sum_{|p - r| \leq 1,\ p \geq q - 3}
  2^{q - \frac{p + r}{2}}\|2^p\Delta_pv\|_{L^\infty}^{\frac{1}{2}}
  \|\Delta_rv\|_{L^2}^{\frac{1}{2}}\|2^p\Delta_pv\|_{L^2}^{\frac{1}{2}}
  \|2^r\Delta_rv\|_{L^2}^{\frac{1}{2}}ds\\\nonumber
&&\leq C\|v\|_{L^\infty_T(L^2)}^{\frac{1}{2}}\|\nabla v
  \|_{L^2_T(L^2)}\|v\|_{\widetilde{L}^1_T(C^1)}^{\frac{1}{2}}.
\end{eqnarray}
Similarly, one has
\begin{eqnarray}\nonumber
&&\int_0^T\Big(\sum_{p - r \geq 2}\|\Delta_q(\Delta_pv
  \otimes \Delta_rv)(s)\|_{L^\infty}
  + \sum_{r - p \geq 2}\|\Delta_q(\Delta_pv \otimes
  \Delta_rv)(s)\|_{L^\infty}\Big)ds\\\nonumber
&&\leq C\int_0^T\sum_{p - r \geq 2,\ | p - q| \leq 2}
  \|\Delta_pv\|_{L^\infty}\|\Delta_rv\|_{L^\infty}ds\\\nonumber
&&\leq C\int_0^T\sum_{p - r \geq 2,\ | p - q| \leq 2}
  \|2^p\Delta_pv\|_{L^\infty}^{\frac{1}{2}}\|2^p\Delta_pv
  \|_{L^2}^{\frac{1}{2}}2^{r - \frac{p}{2}}\|\Delta_rv\|_{L^2}ds\\\nonumber
&&\leq C\|v\|_{L^\infty_T(L^2)}^{\frac{1}{2}}\|\nabla v
  \|_{L^2_T(L^2)}\|v\|_{\widetilde{L}^1_T(C^1)}^{\frac{1}{2}}.
\end{eqnarray}
Using the above two estimate, one can improve \eqref{c5} as
\begin{eqnarray}\nonumber
\|v\|_{\widetilde{L}^1_T(C^1)} \leq C\Big(\sup_q\|\Delta_q
v_0\|_{L^2}\big(1 - e^{-c2^{2q}T}\big) +
\|v\|_{L^\infty_T(L^2)}\|\nabla v\|_{L^2_T(L^2)}^2 +
\|f\|_{\widetilde{L}^1_T(C^{-1})}\Big).
\end{eqnarray}
Consequently, one can deduce \eqref{c2} from and the basic energy
estimate \eqref{c3} and the above inequality.
\end{proof}

Now let us assume that $f \in L^1_T(\dot{C}^{-1}) \cap
L^2_T(H^{-1})$. By Lemma \ref{lem-CM}, it is easy to see that
\begin{eqnarray}\label{c6}
&&\|v\|_{\widetilde{L}^1_{[t_0, T]}(C^1)}\leq
  C\Big(\sup_q\|\Delta_qv(t_0)\|_{L^2}\big(1 - \exp\{-c2^{2q}(T - t_0)\}\\\nonumber
&&\quad +\ \int_{t_0}^T\sup_q\|2^{-q}\Delta_qf(s)\|_{L^\infty}ds +
  \big(\|v(t_0)\|_{L^2} + \|f\|_{L^2_{[t_0, T]}(\dot{H}^{-1})}\big)
  \|\nabla v\|_{L^2_{[t_0, T]}(L^2)}^2\Big)
\end{eqnarray}
holds for any $t_0 \in [0, T)$. By \eqref{c4}, one can choose some
$q_0$ such that
\begin{eqnarray}\nonumber
\sup_{q > q_0}\|\Delta_qv\|_{L^\infty_{[t_0, T]}(L^2)}^2 \leq
\frac{\epsilon}{4C}.
\end{eqnarray}
Furthermore, by the basic energy estimate \eqref{c3}, one can
choose some $t_1 \in [t_0, T)$ such that
\begin{eqnarray}\nonumber
&&\sup_{t_1 \leq t \leq T}\sup_{q \leq q_0}
  \|\Delta_qv(t)\|_{L^2}\big(1 - \exp\{-c2^{2q}(T - t)\}\\\nonumber
&&\leq \sup_{t_1 \leq t \leq T}\|v(t)\|_{L^2}2c2^{2q_0}(T -
  t_1)\\\nonumber
&&\leq  C2^{2q_0}\big(\|v_0\|_{L^2} + \|f\|_{L^2_{[0,
  T]}(\dot{H}^{-1})}\big)(T - t_1) \leq \frac{\epsilon}{4C}.
\end{eqnarray}
Consequently, one has
\begin{eqnarray}\label{c7}
\sup_{t_1 \leq t \leq T}\sup_q\|\Delta_qv(t)\|_{L^2}\big(1 -
\exp\{-c2^{2q}(T - t)\} \leq \frac{\epsilon}{2C}.
\end{eqnarray}
On the other hand, it is obvious that one can choose some $t_2 \in
[t_1, T)$ such that
\begin{eqnarray}\label{c8}
&&\big(\sup_{t_2 \leq t \leq T}\|v(t)\|_{L^2} + \|f\|_{L^2_{[t_2,
  T]}(\dot{H}^{-1})}\big)\|\nabla v\|_{L^2_{[t_2, T]}(L^2)}^2\\\nonumber
&&\quad +\ \int_{t_2}^T\sup_q\|2^{-q}\Delta_qf(s)\|_{L^\infty}ds
  \leq \frac{\epsilon}{2C}.
\end{eqnarray}
Combining \eqref{c7} and \eqref{c8} with \eqref{c6}, one arrives
at
\begin{eqnarray}\label{c9}
\|v\|_{\widetilde{L}^1_{[t_2, T]}(C^1)} \leq \epsilon.
\end{eqnarray}

\bigskip

\textbf{Step 2. H${\rm \ddot{o}}$lder estimate for $v$ and
$\tau$.}

\bigskip

First of all, by \eqref{c9} and the assumption \eqref{a5}, one can
choose $t_\star \in [t_2, T)$ such that
\begin{eqnarray}\label{d1}
\|v\|_{\widetilde{L}^1_{[t_\star, T]}(C^1)} \leq \epsilon,\quad
\|\tau\|_{L^1_{[t_\star, T]}({\rm BMO})} \leq \epsilon.
\end{eqnarray}
For $0 \leq t < T$, define
\begin{eqnarray}\nonumber
A(t) = \sup_{0 \leq s < t}\|v(t, \cdot)\|_{\dot{C}^{1 +
\alpha}},\quad B(t) = \sup_{0 \leq s < t}\|\tau(t,
\cdot)\|_{\dot{C}^{\alpha}}.
\end{eqnarray}
We are about to estimate $A(t)$ and $B(t)$ for $0 \leq t < T$. For
this purpose, let us apply $\Delta_q$ to the Oldroyd-B system
\eqref{a4} to get
\begin{equation}\label{d2}
\begin{cases}
\partial_t\Delta_qv - \Delta \Delta_qv + \nabla\Delta_q p
  = \nabla\cdot\Delta_q(\tau - v\otimes v),\\[-3mm]\\
\partial_t\Delta_q\tau + v\cdot\nabla \Delta_q\tau = \Delta_q\big(\nabla v\tau
  + \tau(\nabla v)^t + D(v)\big) + [v\cdot
  \nabla, \Delta_q]\tau.
\end{cases}
\end{equation}

Let us first estimate $\|v(t, \cdot)\|_{\dot{C}^{1 + \alpha}}$. By
the first equation in \eqref{d2} and Lemma \ref{lem2}, one has
\begin{eqnarray}\label{d3}
&&\|\Delta_qv\|_{L^\infty} \leq Ce^{-c2^{2q}t}
  \|\Delta_qv(0)\|_{L^\infty}\\\nonumber
&&\quad +\ \int_{0}^te^{-c2^{2q}(t - s)} \big\|
  \nabla\cdot\Delta_q(\tau - v\otimes
  v)\big\|_{L^\infty}(s)ds.
\end{eqnarray}
Multiplying $2^{q(1 + \alpha)}$ to both sides of \eqref{d3}, we have
\begin{eqnarray}\nonumber
&&\|\Delta_qv(t, \cdot)\|_{\dot{C}^{1 + \alpha}} \leq
  C\|v(0, \cdot)\|_{\dot{C}^{1 + \alpha}}\\\nonumber
&&\quad +\ C\int_{0}^t2^{2q}e^{-c2^{2q}(t - s)}
  \|\Delta_q\tau\|_{\dot{C}^\alpha}ds\\\nonumber
&&\quad +\ C\int_{0}^t 2^{\frac32 q}e^{-c2^{2q}(t -
  s)}\|(v \otimes v)(s, \cdot)\|_{\dot{C}^{1/2 + \alpha}}ds\\\nonumber
&&\leq C\big(\|v(0, \cdot)\|_{\dot{C}^{1 + \alpha}} +
  B(t)\big) + C\Big(\int_{0}^t\|v\|_{L^4}^4
  \|v\|_{\dot{C}^{1 + \alpha}}^4ds\Big)^{\frac{1}{4}},
\end{eqnarray}
where we have used Holder inequality and the fact that
$\| v \otimes v \|_{\dot{C}^{1/2 + \alpha}}  \leq C
  \|v\|_{L^4}
  \|v\|_{\dot{C}^{1 + \alpha}}$. Consequently, there holds
\begin{eqnarray}\nonumber
&&\|v(t, \cdot)\|_{\dot{C}^{1 + \alpha}}^4 \leq
  C\big(\|v(0, \cdot)\|_{\dot{C}^{1 + \alpha}} +
  B(t)\big)^4 + C\int_{0}^t\|v\|_{L^4}^4
  \|v\|_{\dot{C}^{1 + \alpha}}^4ds\\\nonumber
&&\leq C\big(\|v(0, \cdot)\|_{\dot{C}^{1 + \alpha}} +
  B(\widetilde{t})\big)^4 + C\int_{0}^{t}\|v(s, \cdot)\|_{
  L^2}^2\|\nabla v(s, \cdot)\|_{
  L^2}^2\|v(s, \cdot)\|_{\dot{C}^{1 + \alpha}}^4ds%\\\nonumber
%&&\quad +\ C\int_{0}^{t}\|\nabla v(s, \cdot)\|_{
%  L^2}^2\|v(s, \cdot)\|_{\dot{C}^{1 + \alpha}}^2
%  \ln(2 + \|v(s, \cdot)\|_{\dot{C}^{1 + \alpha}})ds\\\nonumber
%&&\leq C\Big(\|v(0, \cdot)\|_{\dot{C}^{1 + \alpha}}^2 +
%  B(\widetilde{t})^2 + \int_{0}^{t}\big(1 + \|v(s, \cdot)\|_{
%  L^2}^2\\\nonumber
%&&\quad +\ \|\nabla v(s, \cdot)\|_{
%  L^2}^2\big)\|v(s, \cdot)\|_{\dot{C}^{1 + \alpha}}^2
%  \ln(2 + \|v(s, \cdot)\|_{\dot{C}^{1 + \alpha}})ds\Big)
\end{eqnarray}
for any fixed $\widetilde{t}: 0 \leq \widetilde{t} < T$ and $t \leq
\widetilde{t} < T$. Here we used  the fact that $B(t)$ is
nondecreasing. Consequently, Gronwall's inequality gives that
\begin{eqnarray}\nonumber
&&A(\widetilde{t})^4 = \sup_{0 \leq t < \widetilde{t}}\|v(t,
  \cdot)\|_{\dot{C}^{1 + \alpha}}^4\\\nonumber
&&\leq C\big(\|v(0, \cdot)\|_{\dot{C}^{1 + \alpha}} +
  B(\widetilde{t})\big)^4\exp\Big\{C\int_{0}^{t}\|v(s, \cdot)\|_{
  L^2}^2\|\nabla v(s, \cdot)\|_{L^2}^2ds\Big\}.
\end{eqnarray}
Since $\widetilde{t} \in [0, T)$ is arbitrary, using the basic
energy inequality \eqref{aprioriestimate}, we in fact have
\begin{eqnarray}\label{d5}
A(t) \leq C\big(\|v(0, \cdot)\|_{\dot{C}^{1 + \alpha}} +
B(t)\big),\quad 0 \leq t < T.
\end{eqnarray}

Next, by the second equation in \eqref{d2}, we have
\begin{eqnarray}\nonumber
&&\|\Delta_q\tau(t, \cdot)\|_{L^\infty} \leq
  \|\Delta_q\tau(0, \cdot)\|_{L^\infty} +
  \int_{0}^t\Big(2^q\|\Delta_qv(s, \cdot)\|_{L^\infty}\\\nonumber
&& \quad +\ \big\|\Delta_q\big(\nabla v\tau
  + \tau(\nabla v)^t\big)(s, \cdot)\|_{L^\infty} + \|[v\cdot
  \nabla, \Delta_q]\tau(s, \cdot)\big\|_{L^\infty}\Big)ds,
\end{eqnarray}
which implies that
\begin{eqnarray}\label{d6}
&&\|\Delta_q\tau(t, \cdot)\|_{\dot{C}^\alpha} \leq
  C\|\tau(0, \cdot)\|_{\dot{C}^\alpha} +
  \int_{0}^t\Big(\|v\|_{\dot{C}^{1 + \alpha}}\\\nonumber
&& \quad +\ \|\nabla v\|_{L^\infty}\|\tau\|_{\dot{C}^\alpha}
  + \|\tau\|_{L^\infty}\|v\|_{\dot{C}^{1 + \alpha}} + 2^{\alpha q}\|[v\cdot
  \nabla, \Delta_q]\tau(s, \cdot)\big\|_{L^\infty}\Big)ds.
\end{eqnarray}
By Bony's decomposition, one has
\begin{eqnarray}\nonumber
&&[v\cdot \nabla, \Delta_q]\tau = \sum_{|p - q^\prime| \leq 1}
  [\Delta_pv\cdot \nabla, \Delta_q]\Delta_{q^\prime}\tau\\\nonumber
&&\quad +\ \sum_{p \leq q^\prime - 2}[\Delta_pv\cdot \nabla,
  \Delta_q]\Delta_{q^\prime}\tau + \sum_{p \leq q^\prime -
  2}[\Delta_{q^\prime}v\cdot \nabla, \Delta_q]\Delta_p\tau\\\nonumber
&&= \sum_{|q^\prime - q| \leq 2}\big([S_{q^\prime -
  1}v\cdot \nabla, \Delta_q]\Delta_{q^\prime}\tau +
  [\Delta_{q^\prime}v\cdot \nabla, \Delta_q]S_{q^\prime
  - 1}\tau\big)\\\nonumber
&&\quad + \sum_{|p - q^\prime| \leq 1}
  [\Delta_pv\cdot \nabla, \Delta_q]\Delta_{q^\prime}\tau.
\end{eqnarray}
Noting that
\begin{eqnarray}\nonumber
[S_{q^\prime -  1}v, \Delta_q]f = \int h(y)\big[(S_{q^\prime -
1}v)(x) - (S_{q^\prime - 1}v)(x - 2^{-q}y)\big] f(x - 2^{-q}y)dy,
\end{eqnarray}
one has
\begin{eqnarray}\nonumber
\|[S_{q^\prime -  1}v, \Delta_q]f\|_{L^\infty} \leq
C2^{-q}\|\nabla S_{q^\prime -  1}v\|_{L^\infty}\|f\|_{L^\infty}.
\end{eqnarray}
Consequently, we have
\begin{eqnarray}\nonumber
&&\sum_{|q^\prime - q| \leq 2}\int_{0}^t\Big(2^{\alpha q}
  \big\|[S_{q^\prime -  1}v\cdot\nabla, \Delta_q]\Delta_{q^\prime}
  \tau(s, \cdot)\big\|_{L^\infty}\Big)ds\\\nonumber
&&\leq \sum_{|q^\prime - q| \leq 2}\int_{0}^t2^{\alpha q}
  2^{-q}\big\|\nabla S_{q^\prime -  1}v\|_{L^\infty}\|\Delta_{q^\prime}
  \nabla\tau\|_{L^\infty}(s, \cdot)ds\\\nonumber
&&\leq C\sum_{|q^\prime - q| \leq 2}\int_{0}^t\big\|\nabla
  S_{q^\prime -  1}v\|_{L^\infty}[2^{\alpha q^\prime}
  \|\Delta_{q^\prime}\tau\|_{L^\infty}](s, \cdot)ds\\\nonumber
&&\leq C\int_{0}^t\|\nabla v\|_{L^\infty}
  \|\tau\|_{\dot{C}^\alpha}ds.
\end{eqnarray}
Similarly,  we have
\begin{eqnarray}\nonumber
&&\sum_{|q^\prime - q| \leq 2}\int_{0}^t\Big(2^{\alpha q}
  \big\|[\Delta_{q^\prime}v\cdot\nabla, \Delta_q]S_{q^\prime -  1}
  \tau(s, \cdot)\big\|_{L^\infty}\Big)ds\\\nonumber
&&\leq \sum_{|q^\prime - q| \leq 2}\int_{0}^t2^{\alpha q}
  2^{-q}\big\|\Delta_{q^\prime}\nabla v\|_{L^\infty}\|S_{q^\prime -  1}
  \nabla\tau\|_{L^\infty}(s, \cdot)ds\\\nonumber
&&\leq C\int_{0}^t\|\tau\|_{L^\infty}
  \|v\|_{\dot{C}^{1 + \alpha}}ds.
\end{eqnarray}
At last, one computes that
\begin{eqnarray}\nonumber
&&\sum_{|p - q^\prime| \leq 1}\int_{0}^t\Big(2^{\alpha q}
  \big\|[\Delta_{p}v\cdot\nabla, \Delta_q]\Delta_{q^\prime}
  \tau(s, \cdot)\big\|_{L^\infty}\Big)ds\\\nonumber
&&\leq \sum_{p, q^\prime \backsim q} \int_{0}^t
  \Big(2^{(1 + \alpha)q}\big\|[\Delta_{p}v, \Delta_q]\Delta_{q^\prime}
  \tau(s, \cdot)\big\|_{L^\infty}\Big)ds\\\nonumber
&&\quad +\ \sum_{p, q^\prime \geq q + 2} \int_{0}^t
  \Big(2^{(1 + \alpha)q}\big\|[\Delta_{p}v, \Delta_q]\Delta_{q^\prime}
  \tau(s, \cdot)\big\|_{L^\infty}\Big)ds\\\nonumber
&&\leq C\int_{0}^t\|\tau\|_{L^\infty}
  \|v\|_{\dot{C}^{1 + \alpha}}ds.
\end{eqnarray}
The above inequalities  yield an improvement of \eqref{d6}:
\begin{eqnarray}\label{d7}
&&\|\tau(t, \cdot)\|_{\dot{C}^{\alpha}} \leq
  C\|\tau(0, \cdot)\|_{\dot{C}^\alpha}\\\nonumber
&&\quad +\ C\int_{0}^t\big(\|\nabla v\|_{L^\infty} +
\|\tau\|_{L^\infty}\big)\big(\|\tau(s, \cdot)\|_{\dot{C}^{\alpha}}
+ \|v(s, \cdot)\|_{\dot{C}^{1 + \alpha}}\big)ds.
\end{eqnarray}

Now let us insert \eqref{d5} into \eqref{d7} to get
\begin{eqnarray}\nonumber
&&B(t) = \sup_{0 \leq s < t}\|\tau(t, \cdot)
  \|_{\dot{C}^{\alpha}} \leq C\|\tau(0, \cdot)\|_{\dot{C}^\alpha}\\\nonumber
&&\quad +\ C\int_{0}^t\big(\|\nabla v\|_{L^\infty} +
  \|\tau\|_{L^\infty}\big)\big(\|v(0,
  \cdot)\|_{\dot{C}^{1 + \alpha}} + B(s)\big)ds.
\end{eqnarray}
Noting that by the inequalities in Lemma \ref{lem1}, we can estimate
\begin{eqnarray}\nonumber
&&\int_{0}^t\big(\|\nabla v\|_{L^\infty} +
  \|\tau\|_{L^\infty}\big)ds\\\nonumber
&&\leq \int_{0}^{t_\star}\big(\|\nabla v\|_{L^\infty} +
  \|\tau\|_{L^\infty}\big)ds + C\int_{t_\star}^t\big(1 + \|v\|_{L^2} +
  \|\tau\|_{L^2}\big)ds\\\nonumber
&&\quad +\  C\sup_{q}
  \int_{t_\star}^t\|\nabla\Delta_qv\|_{L^\infty}ds\ln\big(e +
  \int_0^t\|v\|_{\dot{C}^{1 + \alpha}}ds\big)\\\nonumber
&&\quad +\ C\int_{t_\star}^t
  \|\tau\|_{{\rm BMO}}\ln(e +
  \|\tau\|_{\dot{C}^{\alpha}})ds\\\nonumber
&&\leq C_\star + C\epsilon\ln\big[e + Ct\big(\|v(0,
  \cdot)\|_{\dot{C}^{1 + \alpha}} + B(t)\big)\big] +
  C\epsilon\ln\big(e + B(t)\big)\\\nonumber &&\leq C_\star +
C\epsilon\ln\big(e + \|v(0,
  \cdot)\|_{\dot{C}^{1 + \alpha}} + B(t)\big).
\end{eqnarray}
Here $C_\star$ is a positive constant depending on the solution $(v,
\tau)$ on $[0, t_\star]$. Consequently, we have
\begin{eqnarray}\nonumber
B(t) \leq C_\star\Big(1 + \|\tau(0, \cdot)\|_{\dot{C}^\alpha} +
\|v(0, \cdot)\|_{\dot{C}^{1 + \alpha}}\Big) +
C\int_{0}^t\big(\|\nabla v\|_{L^\infty} +
  \|\tau\|_{L^\infty}\big)B(s)ds.
\end{eqnarray}
Then Gronwall's inequality yields that
\begin{eqnarray}\nonumber
&&e + \|v(0, \cdot)\|_{\dot{C}^{1 + \alpha}} + B(t)\\\nonumber
&&\leq C_\star\Big(e + \|\tau(0, \cdot)\|_{\dot{C}^\alpha} +
  \|v(0, \cdot)\|_{\dot{C}^{1 + \alpha}}\Big)\exp\Big\{C
  \int_{t_\star}^t\big(\|\nabla v\|_{L^\infty} +
  \|\tau\|_{L^\infty}\big)ds\Big\}\\\nonumber
&&\leq C_\star\Big(e + \|\tau(0, \cdot)\|_{\dot{C}^\alpha} +
  \|v(0, \cdot)\|_{\dot{C}^{1 + \alpha}}\Big)e^{C_\star
  + C\epsilon\ln\big(e + \|v(0, \cdot)\|_{\dot{C}^{1 + \alpha}} + B(t)\big)}\\\nonumber
&&\leq C_\star\Big(e + \|\tau(0, \cdot)\|_{\dot{C}^\alpha} +
  \|v(0, \cdot)\|_{\dot{C}^{1 + \alpha}}\Big)\big(e + \|v(0,
  \cdot)\|_{\dot{C}^{1 + \alpha}} + B(t)\big)^{C\epsilon}.
\end{eqnarray}
From the above inequalities and \eqref{d5}, we have
\begin{eqnarray}\label{d8}
A(t) + B(t) \leq C_\star\Big(1 + \|\tau(0, \cdot)\|_{\dot{C}^\alpha}
+ \|v(0, \cdot)\|_{\dot{C}^{1 + \alpha}}\Big)^2
\end{eqnarray}
by choosing $\epsilon = \frac{1}{2C}$.

\section{Proof of Corollary \ref{thm-criterion-MHD}}

In fact, it is easy to see that the tensor $H\otimes H$ satisfies
the following transport equation:
\begin{eqnarray}\label{k1}
\partial_t(H\otimes H) + v\cdot\nabla(H\otimes H) = \nabla v(H\otimes
H) + (H\otimes H)(\nabla v)^t.
\end{eqnarray}
Hence, the tensor $H\otimes H - \frac{1}{2}I$ plays the role of
$\tau$ (However, it seems not being able to directly apply Theorem
\ref{thm-criterion-Oldroyd} to get Corollary
\ref{thm-criterion-MHD}). The rest part of the proof of Corollary
\ref{thm-criterion-MHD} is similar as that of Theorem
\ref{thm-criterion-Oldroyd}.

In fact, by the assumption $\|v(0, \cdot)\|_{L^2\cap\dot{C}^{1 +
\alpha}} + \|H(0, \cdot)\|_{L^2\cap\dot{C}^\alpha} < \infty$, one
can easily derive that
\begin{eqnarray}\label{k2}
\|v(0, \cdot)\|_{L^2 \cap \dot{C}^{1 + \alpha} (\mathbb{R}^2) } +
\|H\otimes H(0, \cdot)\|_{ L^1 \cap \dot{C}^\alpha (\mathbb{R}^2)} <
\infty.
\end{eqnarray}
Moreover, one has the following energy law
\begin{eqnarray}\label{k3}
\|(v, H)\|_{L^\infty_T(L^2)} + 2\|\nabla v\|_{L^2_T(L^2)} = \|(v,
H)(0, \cdot)\|_{L^2},
\end{eqnarray}
which gives that
\begin{eqnarray}\label{k4}
\|H\otimes H\|_{L^\infty_T(L^1)} < \infty.
\end{eqnarray}
Having \eqref{k2}, \eqref{k3} and \eqref{k4} in hand, and noticing
the assumption \eqref{a9} and the transport equation \eqref{k1} for
$H\otimes H$, one has
\begin{eqnarray}\nonumber
\|v(t, \cdot)\|_{\dot{C}^{1 + \alpha}} + \|H\otimes H(t,
\cdot)\|_{\dot{C}^\alpha} < \infty
\end{eqnarray}
by an exactly same manner as in section 3. Coming back to the
transport equation for $H$ in \eqref{MHD}, we have
\begin{eqnarray}\nonumber
\|H\|_{\dot{C}^\alpha} < \infty
\end{eqnarray}
in a standard manner. This completes the proof of Corollary
\ref{thm-criterion-MHD}.

%----------------------------------------------------------------------------acknowledgement
\section*{Acknowledgement}
Zhen Lei was in part supported by NSFC (grant no. 10801029), a
special Postdoc Science Foundation of China (grant no. 200801175)
and SGST 09DZ2272900. Nader Masmoudi was in part supported by NSF
under Grant DMS-0703145. Yi Zhou is partially supported by the NSFC
(grant no. 10728101), the 973 project of the Ministry of Science and
Technology of China, the Doctoral Program Foundation of the Ministry
of Education of China, the ¡±111¡± project (B08018) and SGST
09DZ2272900.

%\bibliographystyle{plain}
% \bibliography{biblio}
%\end{document}

%-----------------------------------------------------------------------------bibliography

\end{document}